\documentclass[12pt]{article}
\usepackage{amsmath, amsthm, amssymb, tikz-cd, cite, graphicx}
\usetikzlibrary{arrows.meta}
\tikzcdset{arrow style=tikz, diagrams={>=latex}}
\theoremstyle{definition}

\newtheorem*{theorem*}{Theorem}
\newtheorem{proposition}{Proposition}

\newcommand{\RD}{\mathrm{RD}}
\title{Removal of 5 Terms from a Degree 21 Polynomial}
\author{
	Curtis Heberle
}
\date{\today}

\begin{document}
\maketitle

\begin{abstract}In 1683 Tschirnhaus claimed to have developed an algebraic method to determine the roots of any degree $n$ polynomial. His argument was flawed, but it spurred a great deal of work by mathematicians including Bring, Jerrard, Hamilton, Sylvester, and Hilbert. Many of the problems they considered can be framed in terms of the geometric notion of resolvent degree, introduced by Farb and Wolfson. Roughly speaking, we have $RD(n) \leq n-k$ if a general degree $n$ polynomial can be put into an $(n-k)$-parameter form. In the present note we discuss a method introduced by Sylvester for solving systems of polynomial equations, and apply it to finding bounds on resolvent degree. In particular, we prove the bound $RD(21) \leq 15$. This bound has been independently established by Sutherland using the classical theory of polarity.
\end{abstract}

\tableofcontents

\section{Introduction}
Consider a degree $n$ polynomial $p(x)$ with roots $x_1, \ldots, x_n$:
$$p(x) = x^n+a_1x^{n-1}+\ldots+a_nx+a_n = \prod_{i=1}^n (x-x_i).$$
A \emph{Tschirnhaus transformation} is a polynomial transformation of the roots of $p$. That is, given a polynomial
$$T(x) =  b_{n-1}x^{n-1}+b_{n-2}x^{n-2}+\ldots+b_1x+b_0,$$
and setting
$$y = T(x)$$
we can form a new polynomial with roots $T(x_1), \ldots, T(x_n)$:
$$q(y) = y^n+A_1y^{n-1}+\ldots+A_{n-1}y+A_{n}= \prod_{i=1}^n (y-T(x_i)).$$
Tschirnhaus's original idea was to transform $p$ to a solvable form -- specifically, to choose $T$ such that $q(y) = y^n+A_n$ -- as part of a proposed method for determining in radicals the roots of $p$. This method is not successful because in general it is not possible to determine in radicals the coefficients of the necessary $T$. On the other hand, a number of interesting results have since been proved involving the use of Tschirnhaus transformations to reduce families of polynomials to certain canonical forms; in particular, the problem of finding $T$ such that $q$ has the $(n-k)$-parameter form
$$q(y) = y^n+A_{k+1}y^{n-k-1}+\ldots+A_{n-1}y+A_n$$
with $A_1 = A_2 = \ldots = A_k = 0$ has been widely considered. We provide a partial review of the historical development of this theory in Section \ref{background}.

Recently progress has been made by casting the problem in terms of the geometric notion of \emph{resolvent degree}, as introduced by Farb and Wolfson.\cite{farb2020resolvent} We review the precise definitions in Section \ref{rd}, but roughly speaking the idea is as follows. We say a generically finite dominant map $X \to Y$ has \emph{essential dimension} at most $d$ if there is a pullback square
\[
\begin{tikzcd}
X \arrow[rr] \arrow[dd] &  & W \arrow[dd] \\
                        &  &              \\ 
Y \arrow[rr]            &  & Z           
\end{tikzcd}
\]
with $\dim(Z) \leq d$. We then say $X \to Y$ has \emph{resolvent degree} at most $d$ if it factors as a composition of maps each of essential dimension at most $d$. (This is somewhat stronger than what is required in the precise definition; for instance, it is sufficient for $X \to Y$ to be birationally equivalent to such a composition.)

In particular, we are interested in the resolvent degree of the root cover of the family of degree $n$ polynomials. That is, let $\mathcal{P}_n$ be the space of monic degree $n$ polynomials, and $\widetilde{\mathcal{P}_n}$ the space of pairs $(p, \lambda)$ of polynomials $p \in \mathcal{P}_n$ with a choice of root $\lambda$. Then there is a generically finite dominant map $\widetilde{\mathcal{P}_n} \to \mathcal{P}_n$ by ``forgetting the root'', and we are interested in bounds on
$$\RD(n) := \RD(\widetilde{\mathcal{P}_n} \to \mathcal{P}_n).$$
Intuitively, we have $\RD(n) \leq d$ if the roots of any $p \in \mathcal{P}_n$ can be determined by solving algebraic functions of at most $d$ variables. For example, for $n = 2$, the quadratic formula implies that the root cover $\widetilde{\mathcal{P}_2} \to \mathcal{P}_2$ is a pullback of the map
\begin{align*}
\mathbb{P}^1 \to \mathbb{P}^1
z \mapsto z^2.
\end{align*}
Similarly, the existence of ``formulas in radicals'' for $n=3$ and $n = 4$ implies that $\widetilde{\mathcal{P}_3} \to \mathcal{P}_3$ and $\widetilde{\mathcal{P}_4} \to \mathcal{P}_4$ factor into towers of pullbacks of cyclic self-covers of $\mathbb{P}^1$, so that $\RD(3) = \RD(4) = 1$.

Many classical results concerning Tschirnhaus transformations can be translated into statements about resolvent degree: if any $p \in \mathcal{P}_n$ can be reduced to an $(n-k)$-parameter form by means of a Tschirnhaus transformation $T$, we have
$$\RD(n) \leq n-k,$$
provided that $T$ can \emph{itself} always be determined by solving algebraic functions of at most $n-k$ variables. For example, a result of Bring shows that any degree $5$ polynomial can be put into the one-parameter form $y^5+A_5y+1$.\cite{chen2017erland} This is sufficient to imply $\RD(5) = 1$; in Section \ref{rd} we treat this example in detail. The bounds so-obtained are generally not sharp, as the improvements obtained by Wolfson and Sutherland have demonstrated.\cite{wolfson2021tschirnhaus, sutherland2021upper}

On the other hand, this formulation of the problem implies different constraints on which Tschirnhaus transformations $T$ are permissible than were generally assumed by classical authors. For example, in considering the problem of transforming a general degree $n$ polynomial $p(x)$ into an $(n-k)$-parameter form by means of a Tschirnhaus transformation $T$, Sylvester only considers those transformations whose coefficients can be determined without solving any equation of degree higher than $k$. This is much more restrictive than necessary if one's goal is to show
$$\RD(n) \leq n-k.$$

For example, for $k = 6$, Sylvester obtains the following:
\begin{proposition}[Sylvester]
For $n \geq 44$, a general polynomial of degree $n$ can be put into an $(n-6)$-parameter form by means of a Tschirnhaus transformation whose coefficients can be determined without solving any equation of degree higher than 5.
\end{proposition}
As a corollary, one obtains the resolvent degree bound $\RD(n) \leq n-6$ for $n \geq 44$; on the other hand, finding the necessary Tschirnhaus transformation is a resolvent degree $\RD(5) = 1$ problem, since only equations of degree at most 5 are involved. Sylvester's was nonetheless the best known bound for $k = 6$ prior to Wolfson's 2020 proof that $\RD(n) \leq n-6$ for $n \geq 41$. \cite{wolfson2021tschirnhaus} In Wolfson's argument, finding the necessary Tschirnhaus transformation is shown to be at worst a resolvent degree 35 problem. It has been conjectured by Wiman and (separately) Chebotarev that $\RD(n) \leq 15$ for $n \geq 21$, though their proposed arguments contain gaps. \cite{wiman1927anwendung, chebotarev1954problem} Recently, Sutherland has provided a rigorous proof of this bound using the classical theory of polarity.\cite{sutherland2021upper} 

In the current note we use ideas from Sylvester to provide an alternative proof of the $\RD(n) \leq n-6$ for $n \geq 21$ bound. In particular, we show
\begin{theorem*}[Main Theorem]
For $n \geq 21$, a general polynomial of degree $n$ can be put into an $(n-6)$-parameter form by means of a Tschirnhaus transformation whose coefficients can be determined without solving an equation of degree higher than 20.
\end{theorem*}
This implies $\RD(n) \leq n-6$ for $n \leq 21$.

We now give an outline of the remaining sections of the paper. In Section \ref{background} we give a brief historical overview of the classical theory of Tschirnhaus transformations and some of the important results therein. In Section \ref{sylvester} we turn to Sylvester's work in particular and give a description of his ``method of obliteration'': a technique for finding solutions of systems of polynomial equations when the number of variables is large relative to the number of equations, which lies at the heart of his work on Tschirnhaus transformations. In Section \ref{rd} we discuss resolvent degree and consider the $n = 5$ case in detail as an illustration of the relation between resolvent degree and the classical theory of Tschirnhaus transformations. Section \ref{linear} describes a key ingredient of our main result: the use of the ``method of obliteration'' to determine linear subspaces of hypersurfaces. Finally, in Section \ref{actual_proof} we give the proof of our main theorem.

\section{Some Background on Tschirnhaus Transformations}\label{background}

\subsection{Preliminaries}
Let $p(x)$ be a polynomial over a (not necessarily closed) field $K$. Then we have
$$p(x) = x^n+a_1x^{n-1}+\ldots+a_{n-1}x+a_n = \prod_{i=1}^n (x-x_i)$$
for some coefficients $a_1, \ldots, a_n \in k$ and roots $x_1, \ldots, x_n \in \overline{K}$. Fixing some algebraic extension $L$ of $K$, a \emph{Tschirnhaus transformation} is a polynomial $T \in L[x]$,
$$T(x) = b_0+b_1x+\ldots+b_{n-1}x^{n-1}$$
which we apply to the roots of $p$ to obtain a new polynomial
$$q(y) = \prod_{i=1}^n (y-T(x_i)) = y^n+A_1y^{n-1}+\ldots+A_1y+A_0.$$
Note that the coefficients $A_0, \ldots, A_n$ are symmetric functions of the transformed roots $T(x_1), \ldots, T(x_n)$, and hence are symmetric in $x_1, \ldots, x_n$. The coefficients $a_0, \ldots, a_n$ generate the algebra of symmetric polynomials in $x_1, \ldots, x_n$, so we can determine the $A_i$ polynomially in terms of $a_0, \ldots, a_n$ and $b_0, \ldots, b_{n-1}$. In particular, this implies $q \in L[y]$. The precise constraints on the field $L$ from which the coefficients of $T$ may be taken vary from author to author but generally one considers towers of extensions of bounded degree.
Tschirnhaus's original goal was to find a formula in radicals for the roots of $p$. More precisely, taking $L$ to be the solvable closure of $K$, his strategy was to determine a Tschirnhaus transformation $T \in L[x]$ which transforms $p$ into the solvable form $q(y) = y^n+A_n$.  If $y_i$ is a root of $q$, we can find the corresponding root(s) of $p$ (i.e., those $x_j$ such that $T(x_j) = y_i$) as the roots of
$$\mathrm{GCD}(p(x), T(x)-y_i).$$
If there are $d$ roots of $p$ which map to $y_i$, this will be a degree $d$ polynomial. In particular, if $q$ has distinct roots, then the roots of $p$ can be recovered rationally from the roots of $q$. Thus if $q$ is solvable and $T$ has coefficients in a solvable extension of $K$, then $p$ will in general be solvable. (And at worst recovering the roots of $p$ will require the solution of an equation of degree $n-1$.) This strategy fails because it is not generally possible to determine $T$ solvably such that all the intermediate terms $A_1, A_2, \ldots, A_{n-1}$ of $q$ vanish. On the other hand, there are weaker versions of Tschirnhaus's problem which are tractable. More precisely, we can ask: what conditions on $n$ and $L$ are sufficient for there to exist $T$ in $L[x]$ transforming the degree $n$ polynomial $p \in K[x]$ to a polynomial $q \in L[y]$ with $A_1 = \ldots = A_k = 0$? That is, we consider the problem of setting some (rather than all) of the intermediate coefficients to zero (thus reducing the number of parameters of the family of polynomials), and we relax the requirement that $L$ necessarily be a solvable extension.

We now consider in more detail what is necessary to determine $T$ such that $A_1 = \ldots = A_k = 0$. As we remarked above, $A_i$ is a polynomial function of $a_1, \ldots, a_n$ and $b_0, \ldots, b_{n-1}$. Treating the $b_i$'s as unknowns to be determined, we further observe that $A_i$ is homogeneous of degree $i$ in the variables $b_0, \ldots, b_{n-1}$. For example,
\begin{align*}
A_1 &= -\sum_{i=1}^n T(x_i)\\
&= -\sum_{i=1}^n (b_{n-1}x_i^{n-1}+b_{n-2}x_i^{n-2}+\ldots+b_1x_i+b_0)\\
&= \left(-\sum_{i=1}^n x_i^{n-1}\right)b_{n-1}+\left(-\sum_{i=1}^n x_i^{n-2}\right)b_{n-2}+\ldots+\left(-\sum_{i=1}^n x_i\right)b_{1}-nb_0
\end{align*}
we can ensure $A_1 = 0$ as long as the coefficients $b_0, \ldots, b_{n-1}$ of $T$ are chosen to satisfy a single homogeneous linear equation. (The coefficients of this equation are symmetric in the $x_i$, so again can be expressed in terms of the coefficients $a_1, \ldots, a_n$ of $p$; it is not necessary to know the roots $x_i$.) More generally, $A_i$ vanishes if and only if the $i$th symmetric function of $T(x_1), \ldots, T(x_n)$ vanishes, so to find a Tschirnhaus transformation $T$ such that $A_i = 0$ in the transformed polynomial we need to solve a degree $i$ polynomial in $b_0, \ldots, b_{n-1}$ whose coefficients are polynomials in $a_1, \ldots, a_n$. Thus, to determine $T$ we must find a point in $\mathbb{P}^{n-1}_{\mathbf{b}}$ on the intersection of the hypersurfaces $V(A_1), V(A_2), \ldots, V(A_k)$, which are of degree $1, 2, \ldots, k$, respectively.

\subsection{Some Results on Tschirnhaus Transformations from 1683 to Present}
Tschirnhaus, in his 1683 paper introducing these transformations, claimed to be able to solve any degree $n$ polynomial by removing \emph{all} intermediate terms.\cite{von1683methodus} That is, he claimed one could find a transformation $T$ such that
$$A_1 = A_2 = \ldots = A_{n-1} = 0$$
and the transformed polynomial has the solvable form
$$y^n+A_n = 0.$$
The problem with this strategy is that finding the necessary $T$ requires the solution of a system of polynomial equations of degrees $1, 2, \ldots, n-1$. In general this leads to an equation of degree $(n-1)!$, so for $n > 3$ finding $T$ is a priori more complicated than finding the roots of the original degree $n$ polynomial. Tschirnhaus did show explicitly how to find $T$ in the $n = 3$ case, and his proof readily generalizes to removing the first two intermediate terms from any degree $n$ polynomial.

Finding a Tschirnhaus transformation $T$ yielding $A_1 = A_2 = \ldots = A_k = 0$ requires solving a system of $k$ equations in $n$ variables, of degree $1, 2, 3, \ldots, k$. In the worst case this leads to an equation of degree $k!$, but when the system is sufficiently undetermined (i.e., $n$ much larger than $k$), a solution can be found without solving any polynomial of degree greater than $k$. The first result of this kind was given in 1786 by Bring, who showed that when $n \geq 5$, we can find $T$ such that $A_1 = A_2 = A_3 = 0$ without solving any equation of degree higher than three.\cite{chen2017erland}

In 1834, Jerrard recovered independently Bring's result that removal of three terms from a degree 5 polynomial was possible by means of a Tschirnhaus transformation.\cite{Jerrard} In fact, Jerrard went on to claim that his methods could yield a reduction of a general quintic to a \emph{solvable} form.\cite{jerrard1835xxiv} (Jerrard was aware of, but did not accept, Abel's 1824 work on the insolvability of the quintic.) 

Shortly thereafter Hamilton was commissioned by the British Association for the Advancement of Science to investigate the validity of Jerrard's methods, issuing his report in 1837.\cite{hamilton1836inquiry} He showed that Jerrard's reductions were in many cases ``illusory". 
Jerrard allowed Tschirnhaus transformations of degree potentially as high or higher than the original polynomial $p$, and Hamilton demonstrated that in the ``illusory" cases the transformation $T$ found by Jerrard was a multiple of $p$. In this case the transformed polynomial is always $q(y) = y^n$ (the same as if the ``transformation" $y = T(x) = 0$ were permitted) for any degree $n$ polynomial $p$, so there is no hope of determining the roots of $p$ by studying $q$. On the other hand, Hamilton showed that Jerrard's methods can be a made to work --and the ``illusory" cases avoided -- in higher degrees. More precisely, for any fixed $k$, one can find a (nonzero, degree $\leq n-1$) Tschirnhaus transformation $T$ yielding $A_1 = A_2 \ldots = A_k = 0$ without solving any equation of degree higher than $k$, provided the degree $n$ of the original polynomial $p$ is large enough relative to $k$. In particular, for $k = 4$, Hamilton showed $n \geq 11$ suffices, and for $k = 5$, $n \geq 47$. 

In 1886, Sylvester published a geometric explanation of the Hamilton/Jerrard method, sharpened some of the bounds (finding, in particular, $n \geq 10$ suffices for $k = 4$, and $n \geq 44$ for $k = 5$).\cite{sylvester1887so} Sylvester claimed that these sharpened bounds are optimal if no ``elevation of degree" is allowed -- that is, if no equations of degree higher than $k$ are to be solved. This claim requires careful interpretation -- Sylvester does not consider the possibility of higher degree polynomials arising but factoring into lower degree terms. This is known to happen for $k = 3$: if one attempts to remove three terms from a degree $n$ polynomial, the system of equations of degree 1, 2, and 3 leads in general to an equation for degree 6. Sylvester's method is able to avoid this elevation of degree only when $n \geq 5$. For $n = 4$, Sylvester's method is inapplicable, but Lagrange had shown (in 1771, more than 100 years prior) that the degree 6 equation that arises in this case always factors into a pair of cubics, so that the necessary Tschirnhaus transformation can in fact be determined without solving any equation of degree higher than 3.\cite{lagrange1771reflexions}

Further reductions have been achieved, but all loosen or ignore the ``elevation of degree" restriction in one way or another. Viewing the degree 9 polynomial
$$p(x) = a_0x^9+a_1x^8+\ldots+a_8x+a_9$$
as a multi-valued ``algebraic function" sending $(a_0, \ldots, a_9)$ to the set of roots of $p$, Hilbert asked whether this could be written as a composition of algebraic functions of at most 4 variables.\cite{hilbert1927gleichung} As part of his investigation, Hilbert sketched a method for finding a Tschirnhaus transformation such that $A_1 = A_2 = A_3 = A_4 = 0$ when $n \geq 9$. Hilbert's method requires finding a line on a cubic surface. This in turn requires the solution of an equation of degree 27. However, by putting the equation of the cubic surface into a four-parameter canonical form (the ``pentahedral form", originally suggested by Sylvester), Hilbert was able to show that the coefficients of the necessary line were themselves algebraic functions of four variables, so this was sufficient for his purposes. 

In 1945, B. Segre strengthened Hilbert's result, describing an algorithm for finding a Tschirnhaus transformation removing 4 terms from a degree $n \geq 9$ polynomial, without solving any equation of degree higher than 5.\cite{segre1945algebraic} The same result was later proven by Dixmier using different methods.\cite{dixmier1993histoire} Also building on Hilbert's work were Wiman, who sketched an alternative proof of the $n = 9$ bound for $k = 4$, and G. Chebotarev, who extended Wiman's ideas to sketch a proof that $n = 21$ for $k = 5$.\cite{chebotarev1954problem, wiman1927anwendung, sutherland2021gn} In 1975, Brauer showed that, when $n > k!$, a degree $n$ polynomial can be written as a composition of algebraic functions of $(n-k-1)$ variables; in this framing of the problem it is permissible to simply solve the degree $k!$ to find a Tschirnhaus transformation removing $k$ terms, since this will be an algebraic function of fewer variables than the original degree $n$ polynomial.\cite{brauer1975resolvent} For $k\geq 7,$ this $n > k!$ bound is lower than the corresponding bounds for removal of $k$ terms found by Sylvester, though it should again be emphasized that this ignores the ``no elevation of degree`` constraint which is explicit in Sylvester, and implicit in pre-Sylvester work on the problem.


 In 2020, Wolfson described a function $F(k)$ such that a degree $n$ polynomial can be put in an $(n-k-1)$-parameter form whenever $n \geq F(k)$, and which improved significantly on Brauer's bounds.\cite{wolfson2021tschirnhaus} Wolfson frames the problem in terms of the theory \emph{resolvent degree}, which allows bounds on the number of necessary parameters for a family of algebraic functions to be computed based on the dimensions of certain spaces. For $k = 5$, Wolfson shows $n \geq 41$, which also improves on Sylvester's bound of $n \geq 44$. To find the Tschirnhaus transformation accomplishing the reduction in this case requires, among other things, finding a 2-plane inside of a cubic hypersurface in $\mathbb{P}^6$. Informally, since the dimension of the moduli space of cubic hypersurfaces in $\mathbb{P}^6$ is 35, finding a 2-plane inside such a hypersurface is at worst a 35-parameter problem, so is permissible when putting a degree 41 polynomial into 35-parameter form. This approach does not produce the degrees of the equations which must be solved to find the Tschirnhaus transformation (which may be higher than $n$, as in Hilbert's proof that $n = 9$ suffices for $k = 4$, where it was necessary to solve a degree 27 equation to find a line on a cubic surface).

Most recently, Sutherland used a combination of ideas from the classical theory of polarity together with optimizations to the moduli-dimension-counting method to improve on the $F(k)$ bound for $k = 5, \ldots, 15$ and all $k \geq 18$.\cite{sutherland2021upper} In particular, for $k = 5$, Sutherland shows that $n = 21$ suffices, giving the first rigorous proof of the bound first claimed by Chebotarev.

In the current note, we show that the $n = 21$ result can be recovered from Sylvester's method if partial elevation of degree is allowed. More precisely, one can find a Tschirnhaus transformation removing 5 terms (i.e., such that the coefficients of the transformed polynomial satisfy $A_1 = A_2 = A_3 = A_4 = A_5 = 0$) from a degree $n$ polynomial whenever $n \geq 21$, by solving no equation of degree higher than 20. A key ingredient is that there exists an explicit algorithm to find a 2-plane on a cubic hypersurface in $\mathbb{P}^9$, without solving any equation of degree higher than 5.

\section{Sylvester's Obliteration Formula}\label{sylvester}
Though motivated by the problem of finding Tschirnhaus transformations, the procedure Sylvester describes in \cite{sylvester1887so} is remarkably general, allowing for a solution to be found to any sufficiently underdetermined system, without elevation of degree. More precisely, given a system of equations $S$ of polynomial equations in $N$ variables  of degree at most $k$, with $n_i$ the number of equations of degree $i$, Sylvester shows there is a function $l(n_1, \ldots, n_k)$ such that, if $N > l(n_1, \ldots, n_k)$, then a solution to $S$ can be found without solving any equation of degree greater than $k$.

The method is as follows: Pick any $f$ of maximal degree $k$ from the system $S$. We can find a solution to $S$ by first finding a line $L$ contained in the solution set of the subsystem $S' = S \setminus \{f\}$, then intersecting this line with the vanishing locus of $f$. To find $L$, first find a solution $Q = (q_0, \ldots, q_N)$ to $S'$. Then to get the required line it suffices to find a point $P = (p_0, \ldots, p_N)$ such that $P+\lambda Q$ is a solution to $S'$ for all $\lambda \in \mathbb{C}$. For any equation $g \in S'$, then, view $g(P+\lambda Q)$ as a polynomial in $\lambda$; if $\deg(g) = d$, the coefficients of $1, \lambda, \ldots, \lambda^{d-1}, \lambda^d$ must vanish identically. In fact the coefficient of $\lambda^d$ is just $g(Q)$, so vanishes since we chose $Q$ to be a solution of $S'$. The vanishing of the remaining coefficients imposes polynomial conditions on the $p_0, \ldots, p_N$ of degrees $1, 2, \ldots, d$. Ranging over all $g \in S'$, we see that $P$ must be a solution to a system $S''$ with $m_i$ equations of degree $i$, where
\begin{align*}
m_k &= n_k-1\\
m_{k-1} &= n_k+n_{k-1}\\
m_{k-2} &= n_k+n_{k-1}+n_{k-2}\\
&\ \vdots\\
m_1 &= n_k+n_{k-1}+\ldots+n_1.
\end{align*}
Now to find this needed point solution, we proceed inductively: again hold aside some polynomial of highest degree (say $h$), and find a linear solution to the subsystem $S'' \setminus \{h\}$, then intersect that line with $h$, and so on. The number of equations of maximal degree decreases by 1 with each step, so eventually we are left with a (possibly very large) system of linear equations; a line in the solution set to this system can be found provided the number of variables is (at least) one greater than the number of equations. Then to backtrack to a linear solution to the original system requires only for equations of degree $\leq k$ to be solved one at a time.

The core reduction is succinctly summarized by Sylvester's formula of obliteration, which we will refer to repeatedly:
\begin{proposition}[Sylvester]
Given $n_1, \ldots, n_k \in \mathbb{N},$ let
$$[n_k, n_{k-1}, \ldots, n_2, n_1]$$
denote the minimum number of variables such that a linear solution to any system of equations with exactly $n_i$ equations of degree $i$ can be found without elevation of degree. Then
$$[n_k, n_{k-1}, \ldots, n_2, n_1] \leq 1+[m_k, m_{k-1}, \ldots, m_2, m_1]$$
where
$$m_i = \begin{cases} \sum_{j=i}^k n_j & i \neq k \\ n_k-1 & i = k.\end{cases}$$
\end{proposition}
Note that applying the obliteration formula $n_k$ times yields a system with no equations of degree $k$, so all equations of the highest degree can be removed. This can be repeated until only (a large number of) linear equations remain, in which case the minimum number of variables needed is easy to determine. For example, consider the problem of finding a line on a cubic hypersurface. We have a system of equations with 1 equation of degree 3, 0 equations of degree 2, and 0 of degree 1. By repeated application of the obliteration formula,
$$[1, 0, 0] \leq 1+[1, 1] \leq 2+[2] = 2+3 =  5,$$
so it is possible to find a line solvably (in fact, by solving no equation of degree greater than 3) on a cubic hypersurface in $\mathbb{P}^5$.

\section{Resolvent Degree}\label{rd}
\subsection{Definitions}
Let $k$ be an algebraically closed field and let $X \to Y$ be a generically finite dominant map of $k$-varieties. Following Buhler and Reichstein \cite{buhler1999tschirnhaus}, we define the \emph{essential dimension} of this map, $\mathrm{ed}(X \to Y)$, to be the minimal $d$ such that there exists a pullback square
\[
\begin{tikzcd}
X \arrow[rr] \arrow[dd] &  & W \arrow[dd] \\
                        &  &              \\ 
Y \arrow[rr]            &  & Z           
\end{tikzcd}
\]
with $\dim Z \leq d$. The idea of resolvent degree is to extend this notion to allow towers of pullbacks; essentially, a map is of resolvent degree at most $d$ if it can be written as a composition of maps each of essential dimension at most $d$. More precisely, following Farb and Wolfson \cite{farb2020resolvent}, the \emph{resolvent degree}, $\mathrm{RD}(X \to Y)$, is defined to be the minimal $d$ such that there exists a tower
\[ E_r \to \cdots \to E_1 \to E_0 = U, \]
where $U$ is an open subset of $Y$, $E_r \to U$ factors through a dominant map $E_r \to X$, and $\mathrm{ed}(E_i \to E_{i-1}) \leq d$ for each $i$. It follows from these definitions that
\[
\mathrm{RD}(X \to Y) \leq \mathrm{ed}(X \to Y) \leq \dim Y.
\]
It is convenient, in proving results on resolvent degree, to construct towers of maps of bounded resolvent degree (rather than essential dimension). Lemma 2.7 in Farb and Wolfson \cite{farb2020resolvent} shows that this is sufficient.

The connection to polynomials and their roots is as follows. Let $\mathcal{P}_n$ denote the space of monic degree $n$ polynomials with coefficients in $k$ and let
\[
\widetilde{\mathcal{P}_n} = \{(p, r) \in \mathcal{P}_n \times \mathbb{A}^1_k \mid p(r) = 0\},
\]
the space of monic polynomials together with a choice of root. We then define
\[
\mathrm{RD}(n) := \RD(\widetilde{\mathcal{P}_n} \to \mathcal{P}_n)
\]
where the map is ``forgetting the root". $\RD(n)$ captures the complexity of the root cover in the sense that if there is a formula in terms of algebraic functions of at most $d$ variables for finding the roots of a degree $n$ polynomial in terms of its coefficients, then $\RD(n) \leq d$.

\subsection{Resolvent Degree of a Dominant Map}
It is necessary to extend the definition of resolvent degree to dominant maps of $k$-varieties which are not necessarily generically finite. Although we are principally interested in $\RD(n) = \widetilde{\mathcal{P}_n} \to \mathcal{P}_n$, which is generically finite, a number of maps which arise naturally in the study of Tschirnhaus transformations are not. For example, a key step in Bring's proof that a general quintic can be reduced to a one-parameter form involves finding a line on a quadric surface in $\mathbb{P}^3$. Thus we would like to study
$$\mathcal{H}_{2,3}^1 \to \mathcal{H}_{2,3}$$
where $\mathcal{H}_{2, 3}$ is the parameter space of quadric surfaces in $\mathbb{P}^3$ and $\mathcal{H}_{2,3}^1$ is the space of such quadrics together with a choice of line on the surface, and the map is ``forgetting the line''. Since any quadric surface contains an infinite family of lines, this map is not generically finite. We would like to extend the notion of resolvent degree so that $\RD(\mathcal{H}_{2,3}^1 \to \mathcal{H}_{2,3})$ captures the complexity of finding lines on quadric surfaces in the same way that $\RD(\widetilde{\mathcal{P}_n} \to \mathcal{P}_n)$ captures the complexity of finding roots.

This is done in \cite{wolfson2021tschirnhaus} in the following way: given a dominant map $\pi: X \to Y$, define the \emph{resolvent degree} $\RD(X \to Y)$ to be the minimum $d$ for which there exists a dense collection of subvarieties $\{U_\alpha \subset X\}$ with 
$$\pi|_{U_{\alpha}}: U_\alpha \to Y$$ 
a generically finite dominant map and $\RD(U_\alpha \to Y) \leq d$ for all $\alpha.$ (Such a subvariety $U_\alpha$ is called a \emph{rational multi-section} for $\pi$.)

Requiring the collection of multisections to be dense in $X$ is necessary to ensure that resolvent degree behaves well with respect to composition of dominant maps: given $X \to Y$ and $Y \to Z$ dominant, for the composition $X \to Z$ we have $\RD(X \to Z)  \leq \max\{\RD(X \to Y), \RD(Y \to Z)\}$ by \cite[Lemma~4.10]{wolfson2021tschirnhaus}. Having multi-sections $V \to Y$ for $X \to Y$ and $U \to Z$ for $Y \to Z$ is not generally enough to product a multisection for $X \to Z$ (the range of $V \to Y$ could entirely miss $U$). On the other hand, if $V \to Y$ is \emph{surjective}, then we do obtain a multi-section of $X \to Z$. 

Returning to the example of lines on quadrics, an algorithm for determining a line (or finite set of lines) on each quadric surface in terms of algebraic functions of degree $d$ gives a multi-section $U \to \mathcal{H}_{2,3}$ with $\RD(U \to \mathcal{H}_{2,3})\leq \RD(d)$, but this is somewhat less than what is required to show that $\RD(\mathcal{H}_{2,3}^1 \to \mathcal{H}_{2,3}) \leq \RD(d)$. On the other hand, as we shall see in the next section, determining any one line is sufficient to yield a Tschirnhaus transformation reducing the quintic to a one-parameter form, so that $\RD(5) = 1$. 

\subsection{Resolvent Degree of the Quintic}
As an illustrative example, we now give a detailed proof that (as observed in \cite{farb2020resolvent}) Bring's work on Tschirnhaus transformations of quintic polynomials implies $\RD(5)= 1$. (Recall that by definition, $\RD(5) = \RD(\widetilde{\mathcal{P}_5} \to \mathcal{P}_5)$, where $\mathcal{P}_5$ is the parameter space of monic degree 5 polynomials and $\widetilde{\mathcal{P}_5}$ is the space of such polynomials together with a choice of root.)

Recall that given a polynomial $p(x) = x^5+a_1x^4+a_2x^3+a_3x^2+a_4x+a_5$ in $\mathcal{P}_5$ and a Tschirnhaus transformation of the form $T(x) = b_0x^4+b_1x^3+b_2x^2+b_3x+b_4$, we can set $y = T(x)$ and eliminate $x$ to obtain a polynomial 
$$q(y) = y^5+A_1y^4+A_2y^3+A_3y^2+A_4y+A_5,$$
where the coefficient $A_i$ is a degree $i$ homogeneous polynomial in $b_0, b_1, b_2, b_3, b_4$, whose coefficients are integral functions of $a_0, a_1, a_2, a_3, a_4, a_5$.

Our strategy will be to construct a tower of maps
$$X_5 \to X_4 \to X_3 \to X_2 \to X_1 \to \mathcal{P}_5$$
such that the resolvent degree of each map is 1 and there is a dominant rational map $X_5 \to \widetilde{\mathcal{P}_5}$. Informally, in $X_1$ there will be associated to each quintic polynomial $p$ the equation of a hyperplane determining the set of Tschirnhaus transformations such that $A_1 = 0$. In $X_2$, we will further have the information of a quadric $Q$ contained in this hyperplane, corresponding to $A_1 = A_2 = 0$, and the equation of a line $l$ contained in $Q$. In $X_3$ we will further have a choice of Tschirnhaus transformation $T$ which lies on $l$ while also satisfying $A_3 = 0$. Using this information we can construct a map from $X_3$ to the one-dimensional space of quintic polynomials of the form $z^5+Az+1$, and we will take $X_4$ to be the pullback of the root cover of this space, so that $X_4$ associates to $p$ a root of the polynomial obtained by applying $T$ to $p$. Finally, in constructing $X_5$, we recover the roots of $p$ itself from the roots of the transformed polynomial. Each step of this process is of bounded resolvent degree.

More precisely: viewing the $b_i$ as homogeneous coordinates on $\mathbb{P}^4$, there is associated to each polynomial $p \in \mathcal{P}_5$ a hyperplane $H \subset \mathbb{P}^4_{\mathbf{b}}$ consisting of all Tschirnhaus transformations that send $p$ to a polynomial $q(y)$ with $A_1 = 0$. Let $X_1$ denote the space of ordered pairs $(p, H)$ with $p$ and $H$ as just described. Then the map $X_1 \to \mathcal{P}_5$ is birational, hence has resolvent degree 1.

Next, we consider the set of all Tschirnhaus transformations such that $A_1 = 0$ and $A_2 = 0$ in the transformed polynomial. The latter is a degree 2 homogeneous polynomial in $b_0, b_1, b_2, b_3, b_4$, so to each polynomial $p$ there is associated a quadric surface $Q \subset H \cong \mathbb{P}^3$ whose points are Tschirnhaus transformations satisfying both conditions. This defines a map
$$X_1 \to \mathcal{H}_{2,3}$$
where $\mathcal{H}_{2,3}$ is the parameter space of quadric surfaces in $\mathbb{P}^3$. Letting $\mathcal{H}_{2,3}^1$ denote the space of such surfaces together with a choice of line, we consider the map
$$\mathcal{H}_{2,3}^1 \to \mathcal{H}_{2,3}.$$
An algorithm exists for finding a line on a quadric surface that requires solving only a single quadratic equation, so there is a rational multi-section $U \to \mathcal{H}_{2,3}$ with 
$$\RD(U \to \mathcal{H}_{2,3}) \leq \RD(2) = 1.$$
(In fact there is a dense collection of such multi-sections, so $\RD(\mathcal{H}_{2,3}^1 \to \mathcal{H}_{2,3}) = 1$, but this is stronger than we require.)

We define $X_2$ via the pullback square
\[
\begin{tikzcd}
X_2 \arrow[d] \arrow[r]& X_1\arrow[d] \\
U \arrow[r] & \mathcal{H}_{2,3}
\end{tikzcd}
\]
so that $\RD(X_2 \to X_1) \leq \RD(U \to \mathcal{H}_{2,3}) = 1$. Points of $X_2$ are ordered tuples $(p, H, Q, l)$, where $p$ is a quintic polynomial, $H$ is the space of all Tschirnhaus transformations that send $p$ to a polynomial with $A_1 = 0$, $Q \subset H$ is the quadric surface whose points are Tschirnhaus transformations such that $A_1 = 0$ and $A_2 = 0$, and $l$ is a line contained in $Q$.

Finally, to find a Tschirnhaus transformation $T$ satisfying $A_1 = 0$, $A_2 = 0$, and $A_3 = 0$ requires intersecting the cubic hypersurface defined by $A_3 = 0$ with the line $l$. To find the three points of intersection requires only the solution of a cubic equation. Defining $X_3$ to be the space of tuples $(p, H, Q, l, T)$, with $p, H, Q, l$ as above and $T$ a Tschirnhaus transformation yielding $A_1 = A_2 = A_3 = 0$, the projection map
$$X_3 \to X_2$$
has resolvent degree $RD(3) = 1$.

Now let $\mathcal{P}_5'$ denote the space of quintic polynomials of the form $z^5+Az+1$. Given a point $(p, H, Q, l, T)$ in $X_3$, we can apply the transformation $T$ to $p$ to obtain a polynomial
$$q(y) = y^5+A_4y+A_5.$$
Applying a change of variables $y = \sqrt[5]{A_5}z$ and dividing by $A_5$ yields
$$z^5+\frac{A_4\sqrt[5]{A_5}}{A_5}z+1$$
so this procedure defines a map $X_3 \to \mathcal{P}_5'$. Now, the root cover
$$\widetilde{\mathcal{P}_5'} \to \mathcal{P}_5'$$
has resolvent degree
$$\RD(\widetilde{\mathcal{P}_5'} \to \mathcal{P}_5') \leq \dim\left(\mathcal{P}_5'\right) = 1,$$
so defining $X_4$ via the pullback square

\[
\begin{tikzcd}
X_4 \arrow[d] \arrow[r]& X_3\arrow[d] \\
\widetilde{\mathcal{P}_5'} \arrow[r] & \mathcal{P}_5'
\end{tikzcd}
\]
we have $\RD(X_4 \to X_3) = 1$. Points of $X_4$ are of the form $(p, H, Q, l, T, \lambda)$, where $(p, H, Q, l, T) \in X_3$ and $\lambda$ is a root of the transformed polynomial $z^5+\frac{A_4\sqrt[5]{A_5}}{A_5}z+1$. Let $X_5$ be the space whose points are tuples $(p, H, Q, l, T, \mu)$, where $\mu$ is a root of $p$. There is a map
$$X_5 \to X_4$$
defined by
$$(p, H, Q, l, T, \mu) \mapsto \left(p, H, Q, l, T, \frac{T(\mu)}{\sqrt[5]{A_5}}\right).$$

On the other hand, given a root $\lambda$ of the transformed polynomial, we can find the corresponding roots of $p$ by solving the (at worst degree 4) polynomial equation
$$\mathrm{GCD}(p(x), \sqrt[5]{A_5}T(x)-\lambda) = 0$$
so $\RD(X_5 \to X_4) = \RD(4) = 1$. 

In all, we have constructed a tower of maps
$$X_5 \to X_4 \to X_3 \to X_2 \to X_1 \to \mathcal{P}_5$$
each of which has resolvent degree 1. Since the root cover $\widetilde{\mathcal{P}_5} \to \mathcal{P}_5$ factors through the projection $X_5 \to \widetilde{\mathcal{P}_5}$, $(p, H, Q, l, T, \mu) \mapsto (p, \mu)$, we have
$$\RD(5) = \RD(\widetilde{\mathcal{P}_5} \to \mathcal{P}_5) = 1.$$

\section{Linear subspaces of hypersurfaces}\label{linear}
To prove $\RD(5) = 1$ we needed to show one can always find a Tschirnhaus transformation which eliminates the first three intermediate terms of a quintic polynomial. This required finding a solution of a system of three equations of degree 1, 2, and 3, respectively. A key geometric fact which made this tractable was that any quadric surface in $\mathbb{P}^3$ contains a line, and that an equation for such a line can be found by solving a quadratic equation; this allows the degree 2 equation (which determines the quadric surface) to be replaced by a degree 1 equation (which determines the line), so that to find a solution of the system requires only the solution of a cubic equation. 

Similarly, in Hilbert's proof that $\RD(9) = 4$,  a Tschirnhaus transformation removing \emph{four} intermediate terms is required, so a system of equations of degrees 1, 2, 3, and 4 must be solved. Informally, Hilbert's idea is to first find a 3-plane inside the hypersurface determined by the equation of degree 2 (which is possible as long as the ambient dimension is high enough), then to find a line on the cubic surface determined inside this 3-plane by the equation of degree 3. If this can be done, all that remains is to intersect the line with the remaining equation of degree 4 and a solution can be found.

In general, one can consider the problem of finding a $k$-plane inside a degree $d$ hypersurface in $\mathbb{P}^N$. We can ask several questions: 
\begin{enumerate}
\item[Q1:] In terms of $d$ and $k$, how large must the ambient dimension $N$ be to guarantee that any degree $d$ hypersurface contains a $k$-plane?
\item[Q2:] What is the ``resolvent degree of finding the $k$-plane"? That is, if $\mathcal{M}_{d,N}$ is a moduli space of degree $d$ hypersurfaces in $\mathbb{P}^N$ and $\mathcal{M}_{d, N}^k$ is the space of such hypersurfaces together with a choice of $k$-plane, what is the resolvent degree of the map
$$\mathcal{M}_{d,N}^k \to \mathcal{M}_{d, N}$$
which forgets the choice of plane?
\item[Q3:] When is there a constructive algorithm to determine the $k$-plane? What are the degrees of the equations that must be solved? How large must $N$ be if no equation of degree higher than some given bound is to be permitted?
\end{enumerate}

As the $\RD(5) = 1$ and $\RD(9) = 4$ examples above illustrate, the answers to these questions have implications for the problem of finding Tschirnhaus transformations, since replacing a hypersurface with a linear subspace of that hypersurface allows us to replace a degree $d$ equation with one or more equations of degree 1, reducing the total degree of the system to be solved.

An answer to Q1 is given in Debarre and Manivel \cite{debarre1998vari}: for $d > 3$, any degree $d$ hypersurface in $\mathbb{P}^N$ contains a $k$-plane if
$$N \geq \frac{{d+k\choose k}}{k+1}+k.$$
In this case, the map $\mathcal{M}_{d, N}^k \to \mathcal{M}_{d, N}$ is surjective, and an upper bound on its resolvent degree is given by the dimension of the codomain:
$$\RD\left(\mathcal{M}_{d, N}^k \to \mathcal{M}_{d, N}\right) \leq \dim\left(\mathcal{M}_{d, N}\right).$$
For example,
$$\RD\left(\mathcal{M}_{3, 3}^1 \to \mathcal{M}_{3,3}\right) \leq \dim\left(\mathcal{M}_{3,3}\right) = 4,$$
corresponding to Hilbert's observation that finding a line on a cubic surface requires the solution of at most an algebraic function of four variables.

These dimension-counting arguments do not provide enough information to address Q3. For this we return to Sylvester's ideas. First, for the problem of finding a line on a degree $d$ hypersurface, repeated application of the obliteration formula suffices to compute an ambient dimension $N$ such that the desired line may be found without solving any equation of degree higher than $d$. Sylvester's methods can be readily adapted to the problem of finding a $k$-plane on a degree $d$ hypersurface without solving any equation of degree higher than $d$. Note that, when $N$ is large enough for this to be possible,
$$\RD\left(\mathcal{M}_{d,N}^k \to \mathcal{M}_{d, N}\right) \leq \RD(d)$$
giving a bound on resolvent degree which is often sharper than what one obtains from dimension-counting alone (though at the price of requiring a larger ambient dimension $N$ than that given by Waldron's theorem), so Sylvester's ideas are of interest even if one is only concerned with finding bounds on resolvent degree. In the remainder of this section we look at the $d = 3$ case in detail.

\subsection{Linear subspace of cubic hypersurfaces}

Let $\mathcal{M}_{3, N}$ be the moduli space of smooth cubic hypersurfaces in $\mathbb{P}^N$ and let $\mathcal{M}_{3,N}^r$ be the moduli space of pairs $(C, P)$ where $C$ is a smooth cubic hypersurface and $P$ is an $r$-plane contained in $C$. We can then consider the ``resolvent degree of finding an $r$-plane", i.e., 
\[
\RD(\mathcal{M}_{3,N}^r \to \mathcal{M}_{3,N}),
\]
where the map forgets the choice of plane. For example, for the problem of finding a line on a cubic surface, we have
\[
\RD(\mathcal{M}_{3,3}^1 \to \mathcal{M}_{3,3}) \leq \dim(\mathcal{M}_{3,3}) = 4,
\]
so that finding a line on a cubic surface in $\mathbb{P}^3$ requires the solution of algebraic equations of no more than 4 variables. Farb and Wolfson \cite{farb2020resolvent}, using an argument due to Klein, show this bound can be improved to
$$\RD(\mathcal{M}_{3,3}^1 \to \mathcal{M}_{3,3}.) \leq 3.$$

In higher dimensions it is easier to find a line. In particular, by Sylvester's obliteration method, as discussed in section 3, finding a line on a cubic surface in $\mathbb{P}^5$ requires the solution of no equation of degree higher than 3, and so is a resolvent degree 1 problem.
\begin{proposition}[Sylvester]
Given a cubic hypersurface $V$ in $\mathbb{P}^n$, we can find a line contained in $V$ by solving equations of degree no higher than 3 provided $n \geq 5.$ Hence
$$\RD(\mathcal{M}_{3, 5}^1 \to \mathcal{M}_{3, 5}) \leq \RD(3) = 1.$$
\end{proposition}
\begin{proof}
We wish to find a linear solution to a system of equations consisting of 1 equation of degree $3$ and no equations of lower degree. By Sylvester's obliteration formula, the ambient dimension required is
$$[1, 0, 0] \leq 1+[1, 1] \leq 2+[2] = 2+3 = 5.$$
%
\end{proof}
Sylvester's method also extends to finding higher-dimensional linear subspaces of hypersurfaces, when the dimension of the ambient space is large enough. For example, for cubic hypersurfaces in $\mathbb{P}^9$, finding a 2-plane can also be done solvably.
\begin{proposition}[Sylvester]
Given a cubic hypersurface $V = V(f)$ in $\mathbb{P}^{n}$, we can find a 2-plane contained in $V$ by solving equations of degree no higher than 3, provided $n \geq 11$. Hence
$$\RD(\mathcal{M}_{3, 11}^2 \to \mathcal{M}_{3, 11}) \leq \RD(3) = 1.$$
\end{proposition}
\begin{proof}
Since $n \geq 5$, we can find a line $l$ contained in $V \subset \mathbb{P}^n$ solvably. Let $q$ and $r$ be distinct points of $l$ such that $q+\lambda r \in V$ for all $\lambda \in \mathbb{V}$. To find a plane contained in $V$, we look for a point
$$p = \mathbb{P}^{n-2} \subset \mathbb{P}^n \setminus l$$
such that $f(p+\mu q+\lambda r) = 0$ for all $\mu, \lambda \in \mathbb{V}$. Expanding this as a polynomial in $\mu$ and $\lambda$, the coefficients of $1$, $\lambda$, $\mu$, $\lambda\mu$, $\lambda^2$, and $\mu^2$, must vanish identically. 

This gives a system of equations in $n-2$ variables with 1 equation of degree 3, 2 equations of degree 2, and 3 equations of degree 1. To solve it, we look for a linear solution to the subsystem of equations of degree $<3$, then intersect this line with the remaining degree 3 equation. Using Sylvester's formula of obliteration, we have
$$[2, 3] \leq 1+[1, 5] \leq 2+[6] = 2+7 = 9,$$
so we can find the needed linear solution when $n-2 \geq 9$. Thus we can find a line on a cubic hypersurface in $\mathbb{P}^n$ when $n \geq 11$.
\end{proof}
The ambient dimension can be reduced slightly if some elevation of degree is allowed. From Segre [1945, pg 295, Sec 12], it is possible to determine a line on the intersection of two given quadrics by solving equations of degree no higher than 5, provided the ambient dimension is at least 4. Thus to find a linear solution of a system with 2 equations of degree 2, and 3 equations of degree 1, an ambient dimension of 7 is sufficient. Comparing to the last paragraph in the proof above we have
\begin{proposition}
Given a cubic hypersurface $V = V(f)$ in $\mathbb{P}^{n}$, we can find a 2-plane contained in $V$ by solving equations of degree no higher than 5, provided $n \geq 9$. Hence
$$\RD(\mathcal{M}_{3, 9}^2 \to \mathcal{M}_{3, 9}) \leq \RD(5) = 1.$$
\end{proposition}

\section{Removing 5 Terms from a Polynomial}\label{actual_proof}
Given a degree $n$ polynomial
$$x^n+a_1x^{n-1}+\ldots+a_nx+a_n$$
we wish to find a Tschirnhaus transformation
$$T(x) = b_{n-1}x^{n-1}+b_{n-2}x^{n-2}+\ldots+b_1x+b_0$$
such that after setting $y = T(x)$ the transformed polynomial
$$y^n+A_1y^{n-1}+\ldots+A_{n-1}y+A_{n}$$
satisfies $A_1 = A_2 = A_3 = A_4 = A_5 = 0$. To determine the coefficients $b_0, \ldots, b_{n-1}$ of $T$ then requires the solution of a system of equations with degrees 1, 2, 3, 4, 5, in $\mathbb{P}^{n-1}$. In general, to find a solution to such a system requires solving a polynomial of degree $5! = 120$, but when $n$ is large enough this elevation of degree can be partially avoided.

We first informally sketch the geometry underlying Wolfson's bound of $n = 41$. The general idea is that by finding linear subspaces of the hypersurfaces corresponding to the polynomials of our system, the total degree of the system can be reduced. In this case, to find the necessary Tschirnhaus transformation one first finds a $6$-plane inside a quadric surface in $\mathbb{P}^{13}$ (this only requires solving degree 2 equations, so is resolvent degree 1). Then, intersecting the degree 3 equation with the 6-plane yields a cubic hypersurface in $\mathbb{P}^6$. If we can then find a 2-plane inside this cubic hypersurface, then by intersecting the degree 4 and 5 equations with this plane, we are left with a system of total degree 20 to solve. In summary, we have the chain
$$V_4 \cap V_5 \subset \mathbb{P}^2 \subset V_3 \subset \mathbb{P}^{6} \subset V_2 \subset \mathbb{P}^{13} = V_1 \subset \mathbb{P}^{14}.$$

This gives an algorithm for finding the desired Tschirnhaus transformation provided one is able to find the necessary 2-plane inside the cubic hypersurface in $\mathbb{P}^6$. Wolfson uses a dimension count to show that
\[RD(\mathcal{M}_{3,6}^2 \to \mathcal{M}_{3,6}) \leq \dim(\mathcal{M}_{3,6}) = 35,
\]
and so is able to use this Tschirnhaus transformation to show $\RD(n) \leq n-6$ whenever $n-6 \geq 35$.

By increasing the ambient dimension in which the the cubic hypersurface lives, we can use Sylvester's ideas to find a plane inside a cubic hypersurface in $\mathbb{P}^9$ by solving only equations of degree 5 or less. (i.e., in a resolvent degree 1 way). This allows the $n = 41$ bound for removing 5 terms from a degree $n$ polynomial to be reduced to $n = 21$, with simpler irrationalities involved in finding the necesary Tschirnhaus transformation.
\begin{theorem*}[Main Theorem]
Let $n \geq 21$. Given a degree $n$ polynomial
$$x^n+a_1x^{n-1}+\ldots+a_nx+a_n$$
we can find a Tschirnhaus transformation
$$T(x) = b_{n-1}x^{n-1}+b_{n-2}x^{n-2}+\ldots+b_1x+b_0$$
such that after setting $y = T(x)$ the transformed polynomial
$$y^n+A_1y^{n-1}+\ldots+A_{n-1}y+A_{n}$$
satisfies $A_1 = A_2 = A_3 = A_4 = A_5 = 0$. The coefficients $b_0, \ldots b_{n-1}$ of $T$ can be determined by solving equations of degree at most $20$.
\end{theorem*}
\begin{proof}
The equations $A_1 = 0$, $A_2 = 0$, $A_3 = 0$, $A_4 = 0$, $A_5 = 0$ impose polynomial conditions of degree 1, 2, 3, 4, and 5, respectively, on the point $(b_0, \ldots, b_{n-1}) \in \mathbb{P}^{n-1}$ to be determined. Using the degree 1 equation we can eliminate one variable. The degree 2 equation then determines a quadric hypersurface in $\mathbb{P}^{n-2}$. By the classical theory of linear subspaces of quadrics, we can find a $\mathbb{P}^9$ contained in this hypersurface provided $n-2 \geq 19$.

Next, we consider the cubic hypersurface determined by this $\mathbb{P}^9$ and the degree 3 equation. By Proposition 5 of the previous section, we can find a $\mathbb{P}^2$ contained in this hypersurface by solving equations of degree at worst 5. Finally, intersecting the remaining equations of degree 4 and 5 determine two curves in this $\mathbb{P}^2$, whose points of intersection are governed by an equation of degree at most 20. Each such point then satisfies all 5 polynomial conditions, by construction, and so yields a Tschirnhaus transformation of the desired form.
\end{proof}

In the same way that the reduction of the quintic to one-parameter form can be translated into the language of resolvent degree to yield $\RD(5) = 1$, this theorem implies $\RD(n) \leq n-k$ for $n \geq 21$.

\bibliography{tschirnhaus_draft.bib}{}
\bibliographystyle{plain}
\end{document}